\documentclass{amsart}

\usepackage[all]{xy}
\usepackage{amsmath,amssymb}
\usepackage{amsthm}
\usepackage{enumerate}
\usepackage[dvips]{graphicx}
\usepackage{txfonts}
\theoremstyle{plain}
  \newtheorem{theorem}{Theorem}[section]
  \newtheorem{corollary}[theorem]{Corollary}
  \newtheorem{lemma}[theorem]{Lemma}
  \newtheorem{fact}[theorem]{Fact}
  \newtheorem{proposition}[theorem]{Proposition}

\theoremstyle{definition}
  \newtheorem{definition}[theorem]{Definition}
  \newtheorem{example}[theorem]{Example}
  \newtheorem{remark}[theorem]{Remark}

\newtheorem{maina}{Main Theorem}
  
  \newtheorem{mainb}{Main Theorem}
  
  \newtheorem{mainc}{Main Theorem}

\newcommand{\ob}{\mathrm{ob}}
\newcommand{\Cat}{\mathbf{Cat}}

\newcommand{\op}{\mathrm{op}}
\newcommand{\coop}{\mathrm{coop}}

\newcommand{\Bicat}{\mathbf{BiCat}}

\newcommand{\ide}{\mathrm{id}}

\newcommand{\A}{\mathcal{A}}
\newcommand{\B}{\mathcal{B}}

\newcommand{\F}{\mathcal{F}}
\newcommand{\G}{\mathcal{G}}
\newcommand{\E}{\mathcal{E}}
\newcommand{\id}{\mathrm{id}}
\newcommand{\Gr}{\mathrm{Gr}}

\begin{document}

\author{Kohei Tanaka}
\address{Institute of Social Sciences, School of Humanities and Social Sciences, Academic Assembly, Shinshu University, Japan.}
\email{tanaka@shinshu-u.ac.jp}
\title[The Euler characteristic of a bicategory]{The Euler characteristic of a bicategory and the product formula for fibered bicategories}
%\thanks{This work was supported by JSPS KAKENHI Grant Number 15K17535.}
%\renewcommand{\thesection}{\arabic{section}}

\maketitle



\begin{abstract}
This paper studies the Euler characteristic of a bicategory based on the concept of magnitudes introduced by Leinster. We focus on its invariance with respect to biequivalence and on the product formula for Buckley's fibered bicategories.
\end{abstract}

%%%%%%%%%% Section_1 %%%%%%%%%%%%%%%%%%%

\section{Introduction}

The Euler characteristic of a space is well known as a classical topological invariant.
It has two important homotopical properties: an invariance with respect to homotopy equivalence, and the product formula for fibrations as follows.

\begin{fact}\label{fact1.1}
Let $\chi(X)$ denote the Euler characteristic of a suitable space $X$.
\begin{enumerate}
\item If $X$ is homotopy equivalent to $Y$, then $\chi(X)=\chi(Y)$.
\item For a suitable fibration $E \to B$ with fiber $F$ over base $B$ that is path connected, then $\chi(E)=\chi(B)\chi(F)$. 
\end{enumerate}
\end{fact}

Nowadays, Euler characteristic is defined not only for such geometric objects, but also for combinatorial objects: posets \cite{Rot64}, groupoids \cite{BD01}, and categories \cite{Lei08}. 
Leinster showed the properties of the Euler characteristic for a finite category based on Fact \ref{fact1.1} and clarified the relation with the topological Euler characteristic taking classifying spaces \cite{Lei08}.

Furthermore, he developed the theory of Euler characteristics for enriched categories (called {\em magnitudes} in his paper \cite{Lei13}), and the authors examined its homotopical behavior from the viewpoint of model categories \cite{NT16}. Leinster's idea can be naturally extended to bicategories or to more general weak higher categories \cite{GNS}. To define Euler characteristic, it is necessary to consider some finite setting. We deal with {\em measurable} bicategories; the objects are finite, and each category of morphisms is equivalent to a finite category.

This paper focuses on the homotopical properties of the Euler characteristic of a measurable bicategory based on Leinster's results. The notion of {\em biequivalence} is often used in the homotopy theory of bicategories. Indeed, Lack's model structure \cite{Lac04} on the category of bicategories assigns biequivalences to the weak equivalences.

\begin{maina}[Theorem \ref{biequ_Euler}] If measurable bicategories $\A$ and $\B$ are biequivalent to each other, then $\A$ has Euler characteristic if and only if $\B$ has one also, and in that case, $\chi(\A) = \chi(\B)$.
\end{maina}

The theorem above is a generalization for equivalence of categories \cite{Lei08} and for biequivalence of  2-categories \cite{NT16}. This can be easily proved from the proof of Theorem 2.4 in \cite{Lei08}, as also mentioned in Theorem 2.22 of \cite{GNS}.
Main Theorem 1 above and the coherence theorem relate the Euler characteristic for bicategories to that for 2-categories. Moreover, they clarify the relation with the topological Euler characteristics of the classifying spaces for bicategories.

\begin{mainb}[Theorem \ref{classify_bi}]
If $\A$ is a measurable and acyclic bicategory, then it has Euler characteristic and $\chi(\A)=\chi(B\A)$, where $B\A$ is the classifying space of $\A$.
\end{mainb}

On the other hand, the notion of fibrations for bicategories was introduced by 
Street \cite{Str72}, \cite{Str80}, and Buckley \cite{Buc14}.
We propose the concept of {\em (co)fibered in pseudogroupoids} as a special case of Buckley's fibration, and show the product formula with respect to Euler characteristic as follows.

\begin{mainc}[Theorem \ref{main_2}] Suppose that the following measurable bicategories have Euler characteristic. For a lax functor $\E \to \B$ fibered and cofibered in pseudogroupoids with fiber $\F$ over base $\B$ that is connected, then $\chi(\E)=\chi(\B)\chi(\F)$ 
\end{mainc}

The authors have approached the above problem from the viewpoint of the model structure on enriched categories \cite{NT16}. This paper provides another method using the Grothendieck construction for bicategories introduced by Buckley \cite{Buc14}.

The rest of this paper is organized as follows.
Section 2 reviews the theory of Euler characteristic of finite categories, following Leinster's paper \cite{Lei08}.
Section 3 describes our main theorems, including the definitions of and basic facts about bicategories and lax functors. 
We refer to Buckley's paper \cite{Buc14} for the definitions and properties of fibrations and the Gorthendieck construction.

%%%%%%%%%%%% Section_2 %%%%%%%%%%%%%%%%%%

\section{Review of the Euler characteristic of finite categories}

We begin by reviewing the properties of the Euler characteristic of finite categories, referring to Leinster's paper \cite{Lei08}. Almost all the definitions and facts described here were established in his paper.

\begin{definition}
For finite sets $I$ and $J$, an $I\times J$ \textit{matrix} is a function $I\times J\to \mathbb{Q}$. 
For an $I\times J$ matrix $\zeta$ and a $J\times H$ matrix $\eta$, the $I\times H$ matrix $\zeta \eta$ is defined by $\zeta \eta(i,h)=\sum_j \zeta(i,j)\eta(j,h)$ for each $i \in I$ and $h \in H$. 
An $I\times J$ matrix $\zeta$ has a $J\times I$ transpose $\zeta^{\op}$. Given a finite set $I$, we write $u_I: I\to \mathbb{Q}$ (or simply $u$) as the column vector with $u_I(i)=1$ for all $i$ in $I$. Let $\zeta$ be an $I\times J$ matrix. A \textit{weighting} on $\zeta$ is a column vector $k^{*} :J \to \mathbb{Q}$ such that $\zeta k^{*}=u_I$. 
A \textit{coweighting} on $\zeta$ is a row vector $k_{*}:I\to \mathbb{Q}$ such that $k_{*} \zeta =u^{\op}_J$. 
The matrix $\zeta$ {\em has Euler characteristic} if it has a weighting and a coweighting. Then, its {\em Euler characteristic} is defined as
$$|\zeta|=\sum_j k^{j}=\sum_i k_{i} \in \mathbb{Q}.$$
Note that this definition does not depend on the choice of a weighting and a coweighting.
\end{definition}

\begin{remark}\label{regular}
If a similarity matrix $\zeta$ is regular, it always has Euler characteristic.
A weighting $k^{*}$ is given as $\sum_{j} \zeta^{-1}(i,j)$ and a coweighting $k_{*}$ is given as $\sum_{i} \zeta^{-1}(i,j)$.
\end{remark}

For a small category $A$, let $\ob(A)$ denote the set of objects, and for $x,y \in \ob(A)$, let $A(x,y)$ denote the set of morphisms from $x$ to $y$.

\begin{definition}
For a finite category $A$ (whose objects and morphisms are finite), the {\em similarity matrix} 
$\zeta_{A} : \ob(A) \times \ob(A) \to \mathbb{Q}$ is defined by the number of the set of morphisms $A(i,j)$.
We say that $A$ {\em has Euler characteristic} if $\zeta_{A}$ does, and then we define the {\em Euler characteristic} $\chi(A)$ as $|\zeta_{A}|$.
\end{definition}

The definition above is related to the classical Euler characteristic taking classifying spaces of small categories. Here, the classifying space $BA$ of a small category $A$ is the geometric realization of the nerve simplicial set $NA$.

\begin{proposition}[Proposition 2.11 of \cite{Lei08}]\label{classify}
If $A$ is a finite acyclic category (it never has a circuit of morphisms), then it has Euler characteristic and $\chi(A) = \chi(BA)$.
\end{proposition}

Given two small categories $A$ and $B$, we can consider the {\em coproduct} and the {\em product} of them.
The coproduct $A \coprod B$ is defined by $\ob(A) \coprod \ob(B)$ as objects and 
\[
\left(A \coprod B \right)(x,y) = \begin{cases} A(x,y) & x,y \in \ob(A),\\
B(x,y) & x,y \in \ob(B),\\
\emptyset & \textrm{otherwise}.\\
\end{cases}
\] 
On the other hand, the product $A \times B$ is defined by $\ob(A) \times \ob(B)$ as objects and 
\[
\left( A \times B \right) ((a,b), (a',b'))=A(a,a') \times B(b,b').
\]

\begin{proposition}[Proposition 2.6 of \cite{Lei08}]\label{coproduct}
Let $A_{1}, \ldots, A_{n}$ be finite categories having Euler characteristics.
\begin{enumerate}
\item $\chi \left(\coprod_{i=1}^{n}A_{i} \right) = \sum_{i=1}^{n} \chi(A_{i})$.
\item $\chi \left(\prod_{i=1}^{n} A_{i} \right) = \prod_{i=1}^{n} \chi(A_{i})$.
\end{enumerate}
\end{proposition}

As an important property, the Euler characteristic of a finite category is an invariant with respect to equivalence of categories.

\begin{proposition}[Proposition 2.4 of \cite{Lei08}]\label{invariant}
If $A$ and  $B$ are finite categories equivalent to each other, then $A$ has Euler characteristic if and only if $B$ does, and in that case, $\chi(A)=\chi(B)$.
\end{proposition}

The Euler characteristic has another significant property: the product formula for functors fibered and cofibered in groupoids. Fibered functors or fibered categories were originally introduced by Grothendieck for the descent theory \cite{Gro71}. Functors (co)fibered in groupoids are a special case of them.
They behave similarly to topological fibrations (or Kan fibrations in simplicial sets); indeed, a category of small categories admits a model structure having them as fibrations \cite{Tan13}. 

\begin{definition}\label{fibered_functor}
Let $P : E \to B$ be a functor. 
\begin{enumerate}
\item A morphism $f : e \to e'$ in $E$ is called {\em cartesian} if in each $g : e'' \to e'$ and $h : P(e) \to P(e'')$ with $h \circ P(g)=P(f)$, there exists a unique lift $\tilde{h} : e \to e'$ of $h$ satisfying $\tilde{h} \circ g = f$.
\item $P$ is called {\em fibered} if any morphism  $f : b \to P(e)$ in $B$ has a lift $\tilde{f} : e' \to e$ as a cartesian morphism in $E$. We denote $e'$ by $f^{*}(e)$ and call it the {\em pullback} of $e$ by $f$.
\item $P$ is called {\em fibered in groupoids} if every morphism in $E$ is cartesian 
and any morphism $b \to P(e)$ in $B$ has a lift $\tilde{f} : f^{*}(e) \to e$. Obviously, functors fibered in groupoids are a special case of fibered functors.
\end{enumerate}
We can define the dual notions above, {\em cofibered functors} and {\em functors cofibered in groupoids}, by reversing the directions of morphisms.
\end{definition}

Fibered functors are closely related to the notion of Grothendieck construction given as follows.

\begin{definition}
Let $B$ be a small category, and let $F : B^{\op} \to \Cat$ be a lax functor (see Section 3 below) from the opposite category of $B$ to the 2-category $\Cat$ of small categories, functors, and natural transformations.
The {\em Grothendieck construction} $\Gr(F)$ of $F$ is a small category defined as follows:
\begin{itemize}
\item The set of objects consists of pairs of objects $(b,x) \in \ob(B) \times \ob(Fb)$.
\item The set of morphisms $\Gr(F)((b,x),(c,y))$ consists of pairs of morphisms 
$(f,u) \in B(b,c) \times F(c)(x,Ff(y))$.
\end{itemize}

The composition uses natural transformations equipped with the lax functor.
However, composition is not necessary to calculate the Euler characteristic of a category.
It only depends on how many objects and morphisms there are.
\end{definition}

\begin{proposition}[Proposition 2.8 of \cite{Lei08}]\label{gr}
Let $B$ be a finite category having Euler characteristic, and let $F : B^{\op} \to \Cat$ be a lax functor valued in finite categories.
Suppose that $k_{*}$ is a coweighting on $\zeta_{B}$ and that the Grothendieck construction $\mathrm{Gr}(F)$ and each $Fb$ have Euler characteristics.
Then, 
\[
\chi(\Gr(F)) = \sum_{b \in \ob(B)} k_{b} \chi(Fb).
\]
\end{proposition}

Note that the above is the dual statement of Leinster's; he considered a covariant functor $B \to \Cat$ and used weightings.

Let us review some basic facts about functors (co)fibered in groupoids and the Grothendieck construction. For a functor $P : E \to B$ fibered in groupoids, the {\em fiber category} $P^{-1}(b)$ of an object $b$ in $B$ is a groupoid (every morphism is invertible). Here, the fiber category $P^{-1}(b)$ is a subcategory of $E$ consisting of 
$\ob(P^{-1}(b)) = \{e \in \ob(E) \mid P(e)=b\}$ and $P^{-1}(b)(e,e')=\{ f \in E(e,e') \mid Pf=\ide_{b}\}$.
It yields a lax functor $P^* : B^{\op} \to \Cat$ (see Example \ref{fiber_functor} below for the details), and the Grothendieck construction $\Gr(P^*)$ is equivalent to $E$. If $P$ is fibered and cofibered in groupoids, a morphism $f : b \to b'$ in $B$ induces an equivalence of categories $f^{*} : P(b') \to P(b)$.
It implies that any fiber category is equivalent when the base category is connected (any two objects are combined for a zigzag sequence of morphisms); hence, we write it simply as $F$.

\begin{theorem}\label{product_1}
Let a functor $P : E \to B$ be fibered and cofibered in groupoids between finite categories with fiber category $F$ over base $B$ that is connected. If $E$ and $B$ have Euler characteristics, then
\[
\chi(E)=\chi(B)\chi(F).
\]
\end{theorem}
\begin{proof}
Note that the fiber category $F$ is a finite groupoid and has Euler characteristic.
The total category $E$ is equivalent to the Grothendieck construction $\Gr(P^{*})$. By applying Proposition \ref{gr}, we obtain the desired formula:
\[
\chi(E)=\chi(\mathrm{Gr}(P^{*})) = \sum_{b \in \ob(B)} k_{b} \chi(P^{*}b)=\sum_{b \in \ob(B)} k_{b} \chi(F)=\chi(B)\chi(F).
\]
\end{proof}

\begin{corollary}
Let a functor $P : E \to B$ be fibered and cofibered in groupoids between finite categories over $B$ having the connected components $B_{i}$ with the fiber category $F_{i}$. 
If $E$ and each $B_{i}$ have Euler characteristics, then 
\[
\chi(E)=\sum_{i}\chi(B_{i})\chi(F_{i}).
\]
\end{corollary}
\begin{proof}
Let $E_{i}$ denote the full subcategory of $E$ with the set of objects $P^{-1}(\ob(B))$.
The category $E$ is decomposed as $E = \coprod_{i} E_{i}$.
The restriction $P_{|E_{i}} : E_{i} \to B_{i}$ is also fibered and cofibered in groupoids; hence, the desired formula follows from Theorem \ref{product_1} and Proposition \ref{coproduct}:
\[
\chi(E) = \coprod_{i} \chi(E_{i}) = \coprod_{i} \chi(B_{i})\chi(F_{i}).
\]
\end{proof}

%%%%%%%%%%%%% Section_3 %%%%%%%%%%%%%%%%%%%%%

\section{Euler characteristic of bicategories}

This section extends the discussion in Section 2 to the case of bicategories. We need to pay attention to more complicated coherence conditions. 
Before talking about bicategories, we review an essential notion of these called {\em cat-graphs} (\cite{Wol74}, \cite{Lac04}).

\begin{definition}
A {\em cat-graph} $\A$ consists of a set of objects $\ob(\A)$ and a small category $\A(x,y)$ for each pair of objects $x,y$. An object of $\A(x,y)$ is called a {\em 1-morphism} or {\em 1-cell} and is denoted by a single arrow $x \to y$, 
and a morphism of $\A(x,y)(f,g)$ is called a {\em 2-morphism} or {\em 2-cell} and is denoted by a double arrow $f \Rightarrow g$. We call $\A(x,y)$ the {\em category of morphisms}.
\end{definition}

\begin{definition}
A {\em bicategory} $\A$ is a cat-graph equipped with the following data:
\begin{itemize}
\item A composition functor $c : \A(y,z) \times \A(y,x) \to \A(x,z)$ for each triple of objects $x,y,z$.
\item A functor $\ide_{x} : * \to \A(x,x)$ for each object $x$. This can be identified as a 1-morphism $x \to x$.
\item An isomorphism $a_{hgf} : h \circ (g \circ f) \Rightarrow (h \circ g) \circ f$ in $\A(x,w)$ called an {\em associator}, for each composable triple of 1-morphisms $h : z \to w$, $g: y \to z$, and $f : x \to y$.
\item A pair of isomorphisms $l : \ide_{y} \circ f \Rightarrow f$ and $r : f \circ \ide_{x} \Rightarrow f$ in $\A(x,y)$ called  {\em unitors}, for each 1-morphism $f: x \to y$.
\end{itemize}
Therefore, $\A$ is equipped with three types of composition:
\begin{enumerate}
\item The composition of 1-morphisms: $g \circ f=gf : x \to z$ for each $f : x \to y$ and $g : y \to z$.
\item The {\em vertical composition} of 2-morphisms: $\beta \circ \alpha :  f \Rightarrow h$ for each $\alpha : f \Rightarrow g$ and $\beta : g \Rightarrow h$.
\item The {\em horizontal composition} of 2-morphisms: $\alpha * \beta : f' \circ f \Rightarrow g' \circ g$ for each 
$\alpha : f \Rightarrow g : x \to y$ and $\beta : f' \Rightarrow g' : y \to z$. 
\end{enumerate}
These are required to satisfy the coherence condition that makes the following diagrams commute:
\[
\xymatrix{
&  & k(h(gf)) \ar@{=>}[rrd]^{\ide_{k}*a} \ar@{=>}[lld]_{a}& & \\
(kh)(gf)\ar@{=>}[rd]_{a} &&&& k((hg)f) \ar@{=>}[ld]^{a}\\
& ((kh)g)f \ar@{=>}[rr]_{a* \ide_{f}}&& (k(hg))f 
}
\]

\[
\xymatrix{ 
g \circ (\id \circ f ) \ar@{=>}[rr]^{a} \ar@{=>}[rd]_{\ide_{g}*l}& & (g \circ \id) \circ f \ar@{=>}[ld]^{r* \ide_{f}} \\
& g \circ f
}
\]
In particular, a bicategory $\A$ is a 2-category (a category enriched in small categories) when all associators and unitors are identities.
\end{definition}

In the case of 2-categories (more generally, enriched categories), Leinster provided the definition of Euler characteristics (or magnitudes) \cite{Lei13}. 
His definition using weightings and coweightings can be naturally applied to the case of cat-graphs.

\begin{definition}\label{bi_Euler}
A cat-graph $\A$ is called {\em measurable} if it satisfies the following two properties:
\begin{enumerate}
\item The set of objects $\ob(\A)$ is finite.
\item The category of morphisms $\A(x,y)$ is equivalent to a finite category $C$ having Euler characteristic for any $x,y$. We write $\chi(\A(x,y))$ as $\chi(C)$.
\end{enumerate}
For a measurable cat-graph $\A$, the {\em similarity matrix} $\zeta_{\A} :  \ob(\A) \times \ob(\A) \to \mathbb{Q}$ is defined by $\zeta_{\A}(i,j)=\chi(\A(i,j))$.
We say that $\A$ {\em has Euler characteristic} if each $\A(i,j)$ and $\zeta_{\A}$ does, and then we define the {\em Euler characteristic} $\chi(\A)$ as $|\zeta_{\A}|$.
We can define a measurable bicategory and its Euler characteristic for the underlying cat-graph.
\end{definition}

\begin{remark}
The Euler characteristic of a bicategory (or much higher weak category) was introduced in \cite{GNS}.
This is the same definition as ours above.
\end{remark}

Given two cat-graphs $\A$ and $\B$, the (co)product of them is defined by the (co)product of objects and the (co)product of categories of morphisms. We can easily show that the Euler characteristic is compatible with the product and coproduct of the cat-graphs similarly to Proposition 2.4.

\begin{proposition}\label{decompose}
Let $\A_{1}, \ldots, \A_{n}$ be measurable cat-graphs having Euler characteristics.
\begin{enumerate}
\item $\chi \left(\coprod_{i=1}^{n}\A_{i} \right) = \sum_{i=1}^{n} \chi(\A_{i})$.
\item $\chi \left(\prod_{i=1}^{n} \A_{i} \right) = \prod_{i=1}^{n} \chi(\A_{i})$.
\end{enumerate}
\end{proposition}
\begin{proof}
See Proposition 2.6 and Lemma 1.13 of \cite{Lei08}.
\end{proof}

Even when each $\A_{i}$ is a bicategory, the proposition above holds since the (co)product of bicategories has the (co)product of the underlying cat-graphs.

\subsection{Invariance with respect to biequivalence}

{\em Lax functors} or {\em homomorphisms} are known as morphisms between bicategories.

\begin{definition}\label{lax}
For two bicategories $\A$ and $\B$, a {\em lax functor} $L : \A \to \B$ consists of the following data:
\begin{itemize}
\item A function $L : \ob(\A) \to \ob(\B)$.
\item A functor $L_{xy} : \A(x,y) \to \B(Lx,Ly)$ for each pair of objects $x,y$ in $\A$.
\item A natural transformation $\varphi_{xyz} : c \circ (L \times L) \Rightarrow L \circ c : \A(y,z) \times \A(x,y) \to \B(x,z)$ for each triple of objects $x,y,z$ in $\A$. It yields a 2-morphism $\varphi : Lg \circ Lf \Rightarrow L(g \circ f)$ in $\B(x,z)$.
\item A natural transformation $\psi_{x} : \ide_{Lx} \Rightarrow L \circ \ide_{x}: * \to \B(Lx,Lx)$ for each object $x$ in $\A$. It yields a 2-morphism $\psi : \ide_{Lx} \Rightarrow L(\ide_{x})$ in $\B(Lx,Lx)$.
\end{itemize}
These are required to satisfy the coherence condition that makes the following diagrams commute:
\[
\xymatrix{
(Lh \circ Lg) \circ Lf \ar@{=>}[r]^{\varphi*\ide} \ar@{=>}[d]_{a} & L(h \circ g) \circ Lf \ar@{=>}[r]^{\varphi} & L((h \circ g) \circ f) \ar@{=>}[d]^{La} \\
Lh \circ (Lg \circ Lf) \ar@{=>}[r]_{\ide*\varphi} & Lh \circ L(g \circ f) \ar@{=>}[r]_{\varphi} & L(h \circ (g \circ h)) 
}
\]

\[
\xymatrix{
\ide_{Ly} \circ Lf \ar@{=>}[r]^{\psi *\ide} \ar@{=>}[d]_{l} & L(\ide_{y}) \circ Lf \ar@{=>}[d]^{\varphi} && 
Lf \circ \ide_{Lx} \ar@{=>}[r]^{\ide*\psi} \ar@{=>}[d]_{r} & Lf \circ L(\ide_{x}) \ar@{=>}[d]^{\varphi} \\
Lf & L(\ide_{y} \circ f) \ar@{=>}[l]^{L l} && Lf & L(f \circ \ide_{x}). \ar@{=>}[l]^{L r}
}
\]
The above lax functor $L$ is called a {\em pseudofunctor} if all $\varphi$ and $\psi$ are isomorphisms.
In particular, $L$ is a 2-functor when $\A$ and $\B$ are 2-categories and all $\varphi$ and $\psi$ are identities.
\end{definition}

\begin{example}\label{fiber_functor}
The fiber categories yield a pseudofunctor $P^{*} : B^{\op} \to \Cat$ for a fibered functor $P : E \to B$.
Here, we regard the opposite category $B^{\op}$ as a 2-category with only trivial 2-morphisms and $\Cat$ as a 2-category consisting of small categories, functors, and natural transformations.
$P^{*}$ sends an object $b$ to the fiber category $P^{-1}(b)$ and a morphism $f : b \to b'$ to the functor $f^{*} : P^{-1}(b') \to P^{-1}(b)$ induced by pullbacks and cartesian lifts in Definition \ref{fibered_functor}. Note that we have to choose a cartesian lift $\tilde{f} : f^{*}e \to e$ for each object $e$ in $P^{-1}(b')$ to define $f^{*}$, since it is not determined uniquely.
However, it is unique up to isomorphism; hence, $P^{*}$ satisfies the axiom of lax functors.
\end{example}

The notion of {\em cleavages} helps to define $f^{*}$ above.
A {\em cleavage} $\varphi$ for a fibered functor $P : E \to B$ is a map $\varphi(f,e) = \tilde{f} : f^{*}e \to e$ giving a cartesian lift of $f$ at $e$. A fibered functor together with a cleavage is called {\em cloven}.    
For a cloven fibered functor, we can naturally define the functor $f^{*}$ in Example \ref{fiber_functor}.
Every fibered functor has a cleavage; thus, we regard all fibered functors as a cloven for simplicity.

\begin{definition}
A 1-morphism $f : x \to y$ in a bicategory $\A$ is called an {\em equivalence} if there exists $g : y \to x$ such that $g \circ f \cong \ide_{x}$ and $f \circ g \cong \ide_{y}$ in $\A$. 
We write $x \simeq y$ for objects $x,y$ if there exists an equivalence between them. 
We say that $\A$ is a {\em pseudogroupoid} when every 1-morphism is an equivalence and every 2-morphism is an isomorphism. 
\end{definition}

\begin{remark}\label{connectivity}
If two objects $a,b$ are equivalent in a bicategory $\A$, the category of morphisms 
$\A(a,c)$ and $\A(b,c)$ ($\A(c,a)$ and $\A(c,b)$) are equivalent to each other for any object $c$. Therefore, each category of morphisms of a connected pseudogroupoid $\G$ is equivalent to the groupoid $\G(x,x)$ for an object $x$ in $\G$. Here, a bicategory is {\em connected} if any two objects $x,y$ are combined by a zigzag sequence of 1-morphisms $x \rightarrow x_{1} \leftarrow \ldots \rightarrow x_{n} \rightarrow y$. Every bicategory $\A$ can be uniquely decomposed as $\A=\coprod_{i} \A_{i}$ for connected bicategories $\A_{i}$.
\end{remark}

\begin{definition}
Let $L : \A \to \B$ be a lax functor.
\begin{enumerate}
\item $L$ is called {\em biessentially surjective} on objects if for any object $b \in \ob(\B)$ there exists $a \in \ob(\A)$ such that $La \simeq b$ in $\B$. 
\item $L$ is called a {\em local equivalence} if the functor on the categories of morphisms $\A(x,y) \to \B(fx,fy)$ is an equivalence of categories for each $x,y \in \ob(\A)$.
\item $L$ is called a {\em biequivalence} if it is biessentially surjective on objects and a local equivalence.
\end{enumerate}
\end{definition}

An equivalence of categories $F : A \to B$ if and only if there exists an inverse functor $G : B \to A$ such that $G \circ F \cong \ide_{A}$ and $F \circ G \cong \ide_{B}$.
Similarly, a lax functor $L : \A \to \B$ is a biequivalence if and only if there exists an inverse lax functor $K : \B \to \A$ such that $K \circ L \simeq \ide_{\A}$ and $L \circ K \simeq \ide_{\B}$ in the functor bicategories (see \cite{Lei}). Especially, $K$ is also a biequivalence; hence, biequivalence is an equivalence relation on bicategories. In this case, we say that $\A$ and $\B$ are biequivalent to each other.

The Euler characteristic of bicategories is an invariant with respect to biequivalence.
This is a generalization of Leinster's result (Proposition 2.4 in \cite{Lei08}).
It can be easily verified by replacing isomorphisms with equivalences in his proof, as also mentioned in Theorem 2.22 of \cite{GNS}.

\begin{theorem}\label{biequ_Euler}
If measurable bicategories $\A$ and $\B$ are biequivalent to each other, then $\A$ has Euler characteristic if and only if $\B$ does, and in that case, $\chi(\A) = \chi(\B)$.
\end{theorem}
\begin{proof}
Let $L : \A \to \B$ be a biequivalence, and let $l_*$ be a weighting on $\B$.
We denote the equivalence class represented by $x$ as $[x]=\{ y  \mid y \simeq x\}$.
Define $k^{a} = \left( \sum_{b \in [La]} l^{b} \right) /[a]^{\sharp}$, where $[a]^{\sharp}$ is the cardinality of $[a]$. We will show that this is a weighting on $\A$.

If we choose objects $a_{1}, \cdots, a_{m}$ in $\A$ such that $\ob(\A)= \bigcup_{i=1}^{m} [a_{i}]$, then $\ob(\B)=\bigcup_{i=1}^{m}[La_{i}]$ by the biessentially surjectivity of $L$. The desired result follows from the direct calculation:
\begin{align*}
\sum_{a \in \ob(\A)} \zeta_{\A}(a',a)k^{a} &= \sum_{i=1}^{m} \sum_{x \in [a_{i}]} \zeta_{\A}(a',x) k^{a} \\
&= \sum_{i=1}^{m} [a_{i}]^{\sharp} \zeta_{\A}(a',a_{i}) k^{a} \\
&= \sum_{i=1}^{m} \sum_{b \in [La_{i}]} \zeta_{\A}(a',a_{i}) l^{b} \\
&= \sum_{i=1}^{m} \sum_{b \in [La_{i}]} \zeta_{\B}(La',La_{i}) l^{b} \\
&= \sum_{b \in \ob(\B)} \zeta_{\B}(La',b) l^{b}=1.
\end{align*}
Similarly, we can also show the case of coweightings.
In this case, 
\begin{align*}
\chi(\A) = \sum_{a \in \ob(\A)} k^{a} &= \sum_{a \in \ob(\A)} \left(\left( \sum_{b \in [La]} l^{b} \right) /[a]^{\sharp} \right) \\
&=\sum_{i=1}^{m} \sum_{a \in [a_{i}]} \left(\left( \sum_{b \in [La]} l^{b} \right) /[a]^{\sharp} \right) \\
&= \sum_{i=1}^{m} \left( \sum_{b \in [La_{i}]} l^{b} \right) = \sum_{b \in \ob(\B)} l^{b} = \chi(\B).
\end{align*}
\end{proof}

For a bicategory $\A$, there exists a 2-category $S\A$ that is biequivalent to $\A$.
This is known as the {\em coherence theorem} for bicategories. See, for example, Section 2.3 in Leinster's paper \cite{Lei} using the Yoneda embedding, or Lack's paper \cite{Lac04} in terms of model structure.
The 2-category $S\A$ is called the {\em strictification} of $\A$.

\begin{definition}\label{SA}
A bicategory $\A$ is called {\em acyclic} if each category of morphisms of $\A$ is acyclic, and $\A(x,y) = \emptyset$ if $\A(y,x) \neq \emptyset$ when $x \neq y$ and $\A(x,x)$ is equivalent to the trivial category consisting of a single object and the identity. 
\end{definition}

By the definition above, acyclicity is preserved by biequivalences, but measurability is not.
Even if a measurable bicategory $\A$ is biequivalent to $\B$, there is no assurance that the set of objects $\ob(\B)$ is finite. The following lemma examines the case wherein both the conditions acyclicity and measurability are satisfied at the same time.

\begin{lemma}\label{strict}
Let two bicategories $\A$ and $\B$ be biequivalent to each other.
If $\A$ is measurable and acyclic, then so is $\B$.
\end{lemma}
\begin{proof}
It suffices to show the finiteness on $\ob(\B)$.
Let $L : \A \to \B$ be a biequivalence. For any object $b \in \ob(\B)$, there exists $a \in \ob(\A)$ such that $La \simeq b$ in $\B$. Because of the acyclicity of $\B$, these must be equal: $Lb=a$.
The function $\ob(\A) \to \ob(\B)$ on objects is surjective, and $\ob(\B)$ is finite.
\end{proof}

Using these facts, we can relate the Euler characteristic of the classifying space of a bicategory to that of the original bicategory.
Although there are several ways to construct the classifying space of a bicategory \cite{CCG10}, we adopt here the one using the {\em geometric nerves}. 
Recall the classifying space of a small category $C$. The {\em nerve} of $C$ is a simplicial set whose $n$-simplices are functors $[n] \to C$, where $[n]$ is the category (poset) described as $0 \to 1 \to \cdots \to n$. The classifying space $BC$ is defined as the geometric realization of the nerve of $C$.

\begin{definition}
Let $\A$ be a bicategory. 
The {\em geometric nerve} of $\A$ is a simplicial set whose $n$-simplices are lax functors $[n] \to \A$. The {\em classifying space} $B \A$ is the geometric realization of the geometric nerve. 
\end{definition}

If $\A$ is a 2-category, there is another construction of the classifying space using topological categories (categories enriched in topological spaces).
Let $\A_{T}$ denote the topological category with $\ob(\A_{T})=\ob(\A)$ and $(\A_{T})(x,y)=B\A(x,y)$. 

\begin{definition}
Let $T$ be a topological category. The {\em nerve} $NT$ of $T$ is a simplicial space given by 
\[
N_{n}T = \coprod_{x_i \in \ob(T)} T(x_{n-1},x_{n}) \times \cdots T(x_{0},x_1).
\]
The {\em classifying space} $BT$ is defined by the geometric realization of the nerve.
\end{definition}

Bullejos and Cegarra \cite{BC03} proved that these two constructions, $B\A$ and $B\A_{T}$, are homotopy equivalent.

\begin{theorem}[Theorem 1 of \cite{BC03}]
For any 2-category $\A$, there is a natural homotopy equivalence $B\A_{T} \to B\A$.
\end{theorem}

In the case of topological categories, we can consider the notions similar to bicategories.
A topological category $T$ is called {\em acyclic} if $T(x,y) = \emptyset$ and $T(y,x)=\emptyset$ when $x \neq y$, and $T(x,x)$ is the trivial category consisting of a single point. Moreover, $T$ is called {\em measurable} if $\ob(T)$ is finite and each space of morphisms has the homotopy type of a CW complex.
A measurable topological category $T$ has the {\em similarity matrix} $\zeta_{T} : \ob(T) \times \ob(T) \to \mathbb{Q}$ defined by $\zeta_{T}(x,y)=\chi(T(x,y))$ using the topological Euler characteristic.
If $T$ is measurable and acyclic, the classifying space $BT$ has the homotopy type of a CW complex.
The authors have proven the relationship between the classifying spaces of topological categories and the Euler characteristics \cite{NT16}.

\begin{proposition}[Theorem 4.17 of \cite{NT16}]\label{topological}
If $T$ is a measurable and acyclic topological category, then it has Euler characteristic and $\chi(T)=\chi(BT)$.
\end{proposition}

\begin{lemma}\label{2-category}
If $\A$ is a measurable and acyclic 2-category, then it has Euler characteristic and $\chi(\A)=\chi(B\A)$.
\end{lemma}
\begin{proof} 
We can take $\zeta_{\A}$ as a triangular matrix; thus, it is regular and $\A$ has Euler characteristic based on Remark \ref{regular}.
Each space of morphisms of $\B\A$ has the homotopy type of a finite CW complex (simplicial complex).
We can apply Theorem \ref{topological} and obtain the equality $\chi(\A_{T})=\chi(B\A_{T})$, which implies the desired formula:
\[
\chi(\A) = \chi(\A_{T})=\chi(B\A_{T})=\chi(B\A).
\]
Here, we used the facts that $\zeta_{\A}=\zeta_{\A_{T}}$ and $B\A_{T} \simeq B\A$.
\end{proof}

\begin{theorem}\label{classify_bi}
If $\A$ is a measurable and acyclic bicategory, then it has Euler characteristic and $\chi(\A)=\chi(B\A)$.
\end{theorem}
\begin{proof}
By Lemma \ref{strict}, the strictification $S\A$ is measurable and acyclic.
A biequivalence $\A \to S\A$ induces a homotopy equivalence $B\A \to BS\A$ on the classifying spaces.
Lemma \ref{2-category} concludes the result:
\[
\chi(\A) = \chi(S\A) = \chi(BS\A)=\chi(B\A).
\]
\end{proof}

\subsection{Product formula for fibrations}

Next, we focus on the behavior of the Euler characteristic of a bicategory with respect to fibrations.
As we have seen in Section 2, the Euler characteristic of a finite category has a product formula for functors fibered and cofibered in groupoids. A bicategorical version of fibered functors was introduced by Buckley \cite{Buc14} (he called it simply {\em fibration}). We refer the readers to Section 3 of his paper for the details with intelligible diagrams. This section shall describe the definitions and minimum properties necessary to compute Euler characteristics.

\begin{definition}
Let $P : \E \to \B$ be a lax functor. A 1-morphism $f : x \to y$ in $E$ is called {\em cartesian} when it has the following two properties:
\begin{enumerate}
\item For each 1-morphism $g : z \to y$ in $\E$ and $h : Pz \to Px$ in $\B$ with an isomorphism $\alpha :  Pf \circ h \Rightarrow Pg$, 
there exist a 1-morphism $\tilde{h} : z \to x$ in $\E$ and isomorphisms $\tilde{\alpha} : f \circ \tilde{h} \Rightarrow g$ and $\tilde{\beta} : P\tilde{h} \Rightarrow h$ such that $\alpha *(Pf \beta) =P\alpha * \varphi_{hf}$. We say that $(\tilde{h},\tilde{\alpha}, \tilde{\beta})$ is a {\em lift} of $(h,\alpha)$.
\item For a 2-morphism $\sigma : g \Rightarrow g'$ in $\E$ and 1-morphisms $h,h' : Pz \to Px$ in $\B$ with isomorphisms $\alpha : Pf \circ h \Rightarrow Pg$ and $\alpha' : Pf \circ h' \Rightarrow Pg$,  suppose that $(h,\alpha)$ and $(h',\alpha')$ have lifts $(\tilde{h},\tilde{\alpha},\tilde{\beta})$ and $(\tilde{h'},\tilde{\alpha'},\tilde{\beta'})$, respectively. For any 2-morphism $\delta : h \Rightarrow h'$ in $\B$ with $\alpha' * (Pf \delta) = P(\sigma) * \alpha$, there exists a unique 2-morphism $\tilde{\delta} : \tilde{h} \Rightarrow \tilde{h'}$ in $\E$ such that $\tilde{\alpha'} * f \tilde{\delta} = \sigma * \tilde{\alpha}$ and $\delta * \tilde{\beta} = \tilde{\beta'} * P\tilde{\delta}$.
\end{enumerate}
\end{definition}

\begin{definition}
Let $P : \E \to \B$ be a lax functor. A 2-morphism $\alpha : f \Rightarrow g : x \to y$ in $\E$ is called {\em cartesian} if it is cartesian as a morphism in $\E(x,y)$ for the functor $P : \E(x,y) \to \B(Px,Py)$.
We say that $P$ is {\em locally fibered} when $P : \E(x,y) \to \B(Px,Py)$ is fibered for any pair of objects  $x,y$ in $\E$.
\end{definition}

\begin{definition}
A lax functor $P : \E \to \B$ is called {\em fibered} if it satisfies the following conditions:
\begin{enumerate}
\item For any 1-morphism $f : b \to P(e)$ in $\B$, there exists a cartesian 1-morphism $\tilde{f} : e' \to e$ in $\E$ such that $P\tilde{f}=f$.
\item $P$ is locally fibered. 
\item The horizontal composition of any two cartesian 2-morphisms is cartesian.
\end{enumerate}
\end{definition}

Similarly to the case of plane fibered functors, we assume that our fibered lax functor is also equipped with a {\em cleavage}; a cartesian lift is designated for each 1- and 2-morphism.

The following propositions are useful properties with respect to the cartesian morphisms of a fibered lax functor.

\begin{proposition}[Proposition 3.1.12 in \cite{Buc14}]\label{equivalence_preserve}
Let $P : \E \to \B$ be a fibered lax functor and $f$ be a cartesian 1-morphism in $\E$.
If $P(f)$ is an equivalence in $\B$, then so is $f$.
\end{proposition}

\begin{proposition}[Proposition 3.2.1 in \cite{Buc14}]\label{lift_iso}
When $P : \E \to \B$ is locally fibered, every lift $(\tilde{h}, \tilde{\alpha}, \tilde{\beta})$ of $(h, \alpha)$ along a cartesian
1-morphism can be chosen so that $\tilde{\beta}= \ide_{h}$. That is, lifts along cartesian 1-morphisms can be chosen so that $P \tilde{h} = h$.
\end{proposition}

We introduce a bicategorical analog of functors fibered in groupoids as a special case of fibered lax functors.

\begin{definition}
A lax functor $P : \E \to \B$ is called {\em fibered in pseudogroupoids} if it satisfies the following conditions:
\begin{enumerate}
\item Every 1-morphism in $\E$ is cartesian.
\item For any 1-morphism $f : b \to P(e)$ in $\B$, there exists a 1-morphism $\tilde{f} : e' \to e$ in $\E$ such that $P\tilde{f}=f$.
\item $P$ is locally fibered in groupoids.
\end{enumerate}
\end{definition}

In other words, $P : \E \to \B$ is fibered in pseudogroupoids if and only if every $i$-morphism in $\E$ is cartesian and every $i$-morphism in $\B$ has a lift at the target point for $i=1,2$. Obviously, this is a fibered lax functor.

We can define the dual notions above, {\em co-cartesian morphisms}, {\em cofibered lax functors}, and lax functors {\em cofibered in pseudogroupoids}, by reversing the directions of 1- and 2-morphisms.

\begin{definition}
For two lax functors $L, K : \A \to \B$, a lax natural transformation $\sigma : L \Rightarrow K$ consists of the following data:
\begin{itemize}
\item A 1-morphism $\sigma_x : Lx \to Kx$ in $\B$ for each object $x$ in $\A$.
\item A natural transformation $\sigma_{x,y} : (\sigma_{x})^* \circ K \Rightarrow (\sigma_{y})_{*} \circ L: \A(x,y)  \to \B(Lx,Ky)$. It yields a 2-morphism $\sigma_{f} : Kf \circ \sigma_{x} \Rightarrow \sigma_{y} \circ Lf$ in $\B$ for each 1-morphism $f : x \to y$ in $\A$.
\end{itemize}
These are required to make the following diagrams commute:
\[
\xymatrix{
(Kg \circ Kf ) \circ \sigma_{x} \ar@{=>}[r]^{a} \ar@{=>}[d]_{\varphi} & Lg \circ (Kf \circ \sigma_{x}) \ar@{=>}[r]^{\ide * \sigma_{f}} & Kg \circ (\sigma_y \circ Lf) \ar@{=>}[r]^{a^{-1}} & (Kg \circ \sigma_{y}) \circ Lf \ar@{=>}[d]^{\sigma_{g} * \ide} \\
K(g \circ f) \circ \sigma_{x} \ar@{=>}[r]_{\sigma_{gf}} & \sigma_{z} \circ L(g \circ f) \ar@{=>}[r]_{\ide * \varphi^{-1}} & \sigma_{z} \circ (Lg \circ Lf) \ar@{=>}[r]_{a^{-1}} & (\sigma_{z} \circ Lg) \circ Lf
}
\]
\[
\xymatrix{
\ide_{Kx} \circ \sigma_{x} \ar@{=>}[r]^{l} \ar@{=>}[d]_{\psi * \ide} & \sigma_{x} \ar@{=>}[r]^{r^{-1}} & \sigma_{x} \circ \ide_{Lx} \ar@{=>}[d]^{\ide* \phi} \\
K\ide_{x} \circ \sigma_{x} \ar@{=>}[rr]_{\sigma_{\ide}} && \sigma_{x} \circ L \ide_{x}
}
\]
\end{definition}

\begin{definition}\label{fiber_bicategory}
Let $P : \E \to \B$ be a fibered lax functor. The {\em fiber bicategory} $P^{-1}(b)$ over an object $b$ in $\B$ consists of the inverse image by $P$ of $b$, $\ide_{b}$, and $\ide_{\ide_{b}}$ as the objects, 1-morphisms, and 2-morphisms, respectively. Note that this is not a sub-bicategory of $\E$ since a part of the compositions is different from $\E$.
The composition $g \hat{\circ} f$ for a composable pair of 1-morphisms $f,g$ in $P^{-1}(b)$ is defined 
as $\phi^{*}(g \circ f)$: the domain of the cartesian lift $\tilde{\phi}$ of 
\[
\phi=\varphi \circ \theta^{-1} : 1_{b} \Rightarrow 1_{b} \circ 1_{b} =P(g) \circ P(f) \Rightarrow P(g \circ f)
\]
at $g \circ f$.

The vertical composition of 2-morphisms is the same as that of $\E$. However, the horizontal composition $\beta \hat{*} \alpha : f \hat{\circ} 'f \Rightarrow g' \hat{\circ} g$ for a pair of horizontal composable 2-morphisms $\alpha : f \Rightarrow g$ and $\beta : f' \Rightarrow g'$ is defined as the unique 2-morphism in the following left diagram in $\E$ over the right diagram in $\B$:
\[
\xymatrix{
f' \hat{\circ} f \ar@{=>}[r]^-{\tilde{\phi}} \ar@{:>}[d]_{\beta \hat{*} \alpha} & f' \circ f \ar@{=>}[d]^{\beta * \alpha}  &&  \ide_{b} \ar@{=>}[r]^-{\phi} \ar@{=>}[d]_{\ide_{\ide_{b}}} & P(f' \circ f) \ar@{=>}[d]^{P(\beta * \alpha)} \\
g' \hat{\circ} g \ar@{=>}[r]_-{\tilde{\phi}} & g' \circ g && \ide_{b} \ar@{=>}[r]_-{\phi} & P(g' \circ g)
}
\]
Similarly, the identities and coherence isomorphisms are given by the lifting property of $P$.

A 1-morphism $f : b \to b'$ in $\B$ induces a lax functor $f^{*} : P^{-1}(b') \to P^{-1}(b)$ described in the following left diagram in $\E$ over the right diagram in $\B$ (isomorphisms omitted):
\[
\xymatrix{
f^{*}(e) \ar[rr]^{\tilde{f}_{e}} \ar@/_9pt/[d]^{\stackrel{f^{*}\alpha}{\Rightarrow}}_{f^{*}h} \ar@/^9pt/[d]^{f^{*}k} & & e \ar@/_8pt/[d]_{h}^{\stackrel{\alpha}{\Rightarrow}} \ar@/^8pt/[d]^{k} && b \ar[rr]^{f} \ar@/_8pt/[d]^{=}_{\ide_{b}} \ar@/^8pt/[d]^{\ide_{b}} & & b' \ar@/_8pt/[d]_{\ide_{b'}}^{=} \ar@/^8pt/[d]^{\ide_{b'}}\\ 
f^{*}(e') \ar[rr]_{\tilde{f}_{e'}} & & e' && b \ar[rr]_{f} & & b'
}
\]
The lax functor $f^{*}$ sends an object $e$ to the domain $f^{*}e$ of the lift $\tilde{f}_{e}$ and a 1-morphism $h : e \to e'$ to the lift $f^{*}h$ over $\ide_b$ with $\tau_{f} : h \circ \tilde{f}_{e} \cong \tilde{f}_{e'} \circ f^{*}h$ by Proposition \ref{lift_iso}. Furthermore, it sends a 2-morphism $\alpha : h \Rightarrow k$ to the unique lift $f^*\alpha$ over $\ide_{\ide_b}$ that is compatible with the diagram above.
On the other hand, if $P$ is cofibered, the 1-morphism $f : b \to b'$ induces a lax functor $f_{*} : P^{-1}(b) \to P^{-1}(b')$ by reversing the directions of morphisms in the diagram above.

A 2-morphism $\sigma : f \Rightarrow g : b \to b'$ in $\B$ induces a lax natural transformation $\sigma^* : g^* \Rightarrow f^* : P^{-1}(b') \to P^{-1}(b)$ described in the following left diagram in $\E$ over the right diagram in $\B$ (isomorphisms omitted):
\[
\xymatrix{ & f^*e \ar@/^5pt/[dr]^{\tilde{f}_{e}}&  && & b \ar@/^5pt/[dr]^{f}&\\
g^{*}e \ar@/^5pt/[ur]^{\sigma^*_{e}} \ar@/^7pt/[rr]^{g_e}_{\Downarrow \tilde{\sigma}} \ar@/_7pt/[rr]_{\tilde{g}_{e}} & & e && b \ar@/^5pt/[ur]^{\ide} \ar@/^7pt/[rr]^{f}_{\Downarrow \sigma} \ar@/_7pt/[rr]_{g} & & b'
}
\]
The 2-morphism $\tilde{\sigma}$ denotes a lift of $\sigma$ ending at $\tilde{g}_{e}$, and $g_e$ denotes its domain for an object $e$ in $P^{-1}(b')$. There exists a 1-morphism $\sigma_{e}^*$ with an isomorphism $\widetilde{\ide}_{f}$ over the identity $\widetilde{\ide}_{f} : \tilde{f}_{e} \circ \sigma_{e}^* \Rightarrow g_e$. For a 1-morphism $h : e \to e'$, 
the 2-morphism $\sigma^{*}_{h}$ is defined as the unique isomorphism in the following diagram (coherence isomorphisms omitted):
\[
\xymatrix{
\tilde{f}_{e'} \circ f^{*}h \hat{\circ} \sigma_{e}^* \ar@{=>}[r]^{\ide * \tilde{\phi}} \ar@{=>}[d]_{\ide * \sigma_{h}^*} & \tilde{f}_{e'} \circ f^{*}h  \circ \sigma_{e}^* \ar@{=>}[r]^{\tau_f * \ide} & h \circ \tilde{f}_{e} \circ \sigma_{e}^* \ar@{=>}[r]^{\ide * \widetilde{\ide}_{f}} & h \circ g_{e} \ar@{=>}[d]^{\tau_g} \\ 
\tilde{f}_{e'} \circ \sigma_{e'}^* \hat{\circ} g^*h \ar@{=>}[r]_{\ide*\tilde{\phi}} & \tilde{f} \circ \sigma_{e'}^* \circ g^*h \ar@{=>}[rr]_{\widetilde{\ide}_{f}*\ide} && g_{e'} \circ g^*h
}
\]
\end{definition}

We refer the readers to Construction 3.3.5 for more precise descriptions.

\begin{proposition}\label{pro_1}
Let $P : \E \to \B$ be fibered in pseudogroupoids. Then, the fiber bicategory $P^{-1}(b)$ is a pseudogroupoid for each object $b$ in $\B$.
\end{proposition}
\begin{proof}
Recall that $P$ is locally fibered in groupoids, and the vertical composition of $P^{-1}(b)$ coincides with that of $\E$. Thus, every 2-morphism of $P^{-1}(b)$ is invertible. 
Furthermore, Proposition \ref{equivalence_preserve} guarantees that every 1-morphism is an equivalence since $P$ sends it to the identity $\ide_{b}$.
\end{proof}

\begin{proposition}\label{pro_2}
If $P : \E \to \B$ is fibered and cofibered in pseudogroupoids, 
then $f^{*} : P^{-1}(b') \to P^{-1}(b)$ is a biequivalence for a 1-morphism $f : b \to b'$ in $\B$.
\end{proposition}
\begin{proof}
We first show the biessential surjectivity on objects. For an object $e$ in $P^{-1}(b)$, consider the object $f^{*}(f_{*}e)$ in $P^{-1}(b)$. Note that $f_{*}e$ is equipped with a lift $e \to f_{*}e$ of $f$ starting at $e$ and that $f^{*}(f_{*}e)$ is equipped with a lift $f^{*}(f_{*}e) \to f_{*}e$ of $f$ ending at $f_{*}e$.
\[
\xymatrix{
e \ar@{.>}[rr]^{\eta_{e}} \ar[dr]_{} & & f^{*}(f_{*}e) \ar[ld]^{} &
b \ar[rr]^{\ide_{b}} \ar[dr]_{f} & & b \ar[ld]^{f}
\\ 
& f_{*}e && &   b'
}
\]
There exists a lift $\eta_{e} : e \to f^{*}(f_{*}(e))$ of the identity that makes the above left diagram in $\E$ over the right diagram in $B$ commute up to isomorphism. This is an equivalence by Proposition \ref{equivalence_preserve}.

Next, we consider the functor $f^{*} : P^{-1}(b')(e,e') \to P^{-1}(b)(f^*e,f^*e')$ for objects $e,e'$ in $P^{-1}(b)$ and show that this is essentially surjective. For a 1-morphism $h \in P^{-1}(b)(f^{*}e,f^{*}e')$, consider the following left diagram in $\E$ over the right diagram in $\B$:
\[
\xymatrix{
f^{*}e \ar@/^5pt/[r]^{h} \ar@/_5pt/[r]_{f^*k} \ar[d]_{\tilde{f}} & f^*e' \ar[d]^{\tilde{f}} && b \ar@{=}[r] \ar[d]_{f} & b \ar[d]^{f} \\
e \ar[r]_{k} & e' && b' \ar@{=}[r] & b'
}
\]
There exists a 1-morphism $k : e \to e'$ over the identity on $b$, which makes the diagram commute up to isomorphism. The 1-morphism $f^{*}k : f^{*}e \to f^{*}e'$ also makes the diagram commute up to isomorphism. There exists a unique 2-morphism (isomorphism) $h \Rightarrow f^{*}k$ over the identity on $\ide_{b}$.

Finally, we show that $f^*$ is fully faithful.
Given a 2-morphism  $\beta : f^*h \Rightarrow f^*k : f^*e \to f^*e'$ for $h ,k : e \to e'$, consider the following diagrams:
\[
\xymatrix{
f^{*}(e) \ar[rr]^{\tilde{f}_{e}} \ar@/_9pt/[d]^{\stackrel{\beta}{\Rightarrow}}_{f^{*}h} \ar@/^9pt/[d]^{f^{*}k} & & e \ar@/_8pt/[d]_{h}^{\stackrel{\beta_*}{\Rightarrow}} \ar@/^8pt/[d]^{k} && b \ar[rr]^{f} \ar@/_8pt/[d]^{=}_{\ide_{b}} \ar@/^8pt/[d]^{\ide_{b}} & & b' \ar@/_8pt/[d]_{\ide_{b'}}^{=} \ar@/^8pt/[d]^{\ide_{b'}}\\ 
f^{*}(e') \ar[rr]_{\tilde{f}_{e'}} & & e' && b \ar[rr]_{f} & & b'
}
\]
There exists a unique 2-morphism $\beta_* : h \Rightarrow k$ over the identity on $\ide_{b}$ that is compatible with the left diagram. The universality of lifts shows that $f^*(\beta_{*})=\beta$. Similarly, $(f^*\alpha)_*=\alpha$ holds for each $\alpha : h \Rightarrow k$ by the universality of lifts again. 
It completes the result.
\end{proof}

The bicategorical Grothendieck construction was introduced in Construction 3.3.3 of \cite{Buc14}.
He defined it for general {\em trihomomorphisms} using the tricategory structure on bicategories with lax functors, lax natural transformations, and modifications (see \cite{Gur06} for the details of tricategories and trihomomorphisms). For a bicategory $\B$, a trihomomorphism $F : \B^{\coop} \to \Bicat$ consists of a bicategory $Fb$ for each object $b$, a lax functor $f^*$ for each 1-morphism $f$, and a lax natural transformation $\alpha^*$ for each 2-morphism $\alpha$ in $\B$. To describe the precise definition, we require much more complicated coherence axioms than lax functors. 

However, we do not need such higher coherence structures to compute the Euler characteristic.
These are required to assign the compositions and coherence isomorphisms (associators and unitors) to the Grothendieck construction, making it a bicategory. The Euler characteristic of a bicategory only depends on the underlying cat-graph.

For this reason, we shall introduce the notion of trihomomorphisms and the Grothendieck construction for cat-graphs.

\begin{definition}
Let $\B$ be a cat-graph and let $\B^{\coop}$ denote the cat-graph with $\ob(\B^{\coop})=\ob(\B)$ and $\B^{\coop}(x,y)=\B^{\op}(y,x)$.
A {\em trihomomorphism} $F : \B^{\coop} \to \Bicat$ consists of the following data:
\begin{itemize}
\item A bicategory $Fb$ for each object $b$ in $\B$.
\item A lax functor $f^* : Fb' \to Fb$ for a 1-morphism $f : b \to b'$ in $\B$.
\item A lax natural transformation $\alpha^* : g^* \Rightarrow f^*$ for a 2-morphism $\alpha : f \Rightarrow g$ in $\B$.
\end{itemize}
\end{definition}

For example, for a fibered functor $P : \E \to \B$, the fiber bicategories in Definition \ref{fiber_bicategory} yield a trihomomorphism  $P^{*} : \B^{\coop} \to \Bicat$.

\begin{definition}\label{grothen_cat}
Let $F : \B^{\coop} \to \Bicat$ be a trihomomorphism.
The {\em Grothendieck construction} $\Gr(F)$ is a cat-graph consisting of the following data:
\begin{itemize}
\item An object is a pair $(b,x)$ of $b \in \ob(\B)$ and $x \in \ob(Fb)$.
\item A 1-morphism $(b,x) \to (c,y)$ is a pair $(f,u)$ of $f : b \to c$ in $\B$ and of $u : x \to f^*y$ in $Fb$. 
\item A 2-morphism $(f,u) \Rightarrow (g,v) : (b,x) \to (c,y)$ is a pair $(\alpha,\beta)$ of  $\alpha : f \Rightarrow g$ in $\B$ and of $\beta : u \Rightarrow \alpha^*_{y} \circ v$ in $Fb$.
\end{itemize}
\end{definition}

Now, we start to calculate the Euler characteristic of the Grothendieck construction.

\begin{lemma}\label{lemma_1}
Let $\B$ be a measurable cat-graph, and let $F : \B^{\coop} \to \Bicat$ be a trihomomorphism valued in measurable bicategories. 
If we have coweightings on $\zeta_{\B(b,c)}$ and each $\zeta_{Fb(x,f^*y)}$ is all written as $k_*$, then there exists a coweighting on $\Gr(F)((b,x),(c,y))$ defined by $k_{(f,u)}=k_{f}k_{u}$.
\end{lemma}
\begin{proof}
The direct calculation shows the result as follows:
\begin{align*}
&\sum_{(f,u)}k_{f}k_{u} \zeta_{\Gr(F)((b,x),(c,y))}((f,u),(g,v)) \\
=& \sum_{(f,u)} k_fk_u \sum_{\alpha : f \Rightarrow g} \zeta_{Fb(x,f^*y)} (u, \alpha^*_{y} \circ v) \\
=& \sum_{f} k_f \sum_{\alpha : f \Rightarrow g} \sum_{u : x \to f^*y} k_u \zeta_{Fb(x,f^*y)}(u,\alpha^*_y \circ v) \\
=& \sum_{f} k_{f} \zeta_{\B(b,c)}(f,g) =1.
\end{align*}
\end{proof}

The lemma above used coweightings for plain categories. 
Note that the coweightings appearing in the next lemma are for cat-graphs introduced in Definition \ref{bi_Euler}.

\begin{lemma}\label{lemma_2}
Let $\B$ be a measurable cat-graph, and let $F : \B^{\coop} \to \Bicat$ be a trihomomorphism valued in measurable bicategories. Suppose the similarity matrices $\zeta_{\B}$, $\zeta_{Fb}$, and $\zeta_{\Gr(F)}$ all exist.
If we have coweightings on $\zeta_{\B}$ and each $\zeta_{Fb}$ is all written as $k_*$, then there exists a coweighting on $\Gr(F)$ defined by $k_{(b,x)}=k_{b}k_{x}$.
\end{lemma}
\begin{proof}
By Lemma \ref{lemma_1}, the following equality follows from the direct calculation:
\begin{align*}
\sum_{(b,x)} k_bk_x \zeta_{\Gr(F)}((b,x),(c,y)) &= \sum_{(b,x)} k_bk_x \chi(\Gr(F)((b,x),(c,y))) \\
&= \sum_{(b,x)} k_bk_x \sum_{(f,u)} k_{(f,u)} \\
&=\sum_{(b,x)} k_bk_x \sum_{(f,u)} k_fk_u \\
&=\sum_{(b,x)} k_bk_x \sum_{f :  b \to c} k_f \chi(Fb(x,f^*y)) \\
&=\sum_{b} k_b \sum_{f : b \to c} k_f \sum_{x \in \ob(Fb)} k_x \zeta_{Fb}(x,f^*y) \\
&=\sum_{b} k_{b} \sum_{f : b \to c} k_f \\
&=\sum_{b} k_b \chi(\B(b,c)) \\
&=\sum_{b} k_{b} \zeta_{\B}(b,c) = 1.
\end{align*}
\end{proof}

The following proposition is a generalization of Leinster's result (Proposition \ref{gr}) for bicategories.

\begin{proposition}\label{gr_2}
Let $P : \E \to \B$ be a fibered lax functor between measurable bicategories having Euler characteristics.
Let $k_{*}$ be a coweighting on $\zeta_{\B}$, and let each $P^{-1}(b)$ have Euler characteristic.
Then,
\[
\chi(\Gr(P^{*})) = \sum_{b \in \ob(\B)} k_{b} \chi(P^{-1}(b)).
\]
\end{proposition}
\begin{proof}
The desired formula follows from Lemma \ref{lemma_2}:
\[
\chi(\Gr(P^{*})) =\sum_{(b,x)} k_{(b,x)}= \sum_{(b,x)} k_{b}k_{x} = \sum_{b \in \B_{0}} k_{b} \chi(P^{-1}(b)).
\]
\end{proof}

\begin{lemma}\label{G}
A measurable pseudogroupoid $\G$ always has Euler characteristic.
\end{lemma}
\begin{proof}
We may assume that $\G$ is connected. Fix an object $g$ in $\G$. 
Remark \ref{connectivity} implies that each $\zeta_{\G}(x,y)$ coincides with $\zeta_{\G}(g,g)=\chi(\G(g,g))$.
It has a weighting $k^*$ and a coweighting $k_*$ given by 
\[
k^*(x)=k_*(x)=\frac{1}{\chi(\G(g,g)) \cdot \ob(\G)^{\sharp}},
\] 
where $\ob(\G)^{\sharp}$ is the number of the objects of $\G$.
This shows that $\chi(\G)=1/\chi(\G(g,g))$.
\end{proof}

Now, let us recall Buckley's work on the relation between fibered lax functors and Grothendieck constructions.
For a fibered lax functor $P : \E \to \B$,
he showed that $\Gr(P^{*})$ is biequivalent to $\E$ as mentioned in Proposition 3.3.11 of his paper \cite{Buc14}. 
He revealed more complicated structures and coherence conditions on $P^*$ as a trihomomorphism for bicategories and established the Grothendieck construction $\Gr(P^{*})$ as a bicategory. 
Note that the underlying cat-graph of Buckley's Grothendieck construction coincides with our definition in Definition \ref{grothen_cat}.
The lemma below immediately follows from his result and Theorem \ref{biequ_Euler}.

\begin{lemma}\label{Buck}
Let $P : \E \to \B$ be a fibered lax functor between measurable bicategories.
If $\E$ has Euler characteristic, then $\chi(\Gr(P^*))=\chi(\E)$.
\end{lemma}

\begin{theorem}\label{main_2}
Let $P : \E \to \B$ be fibered and cofibered in pseudogroupoids between measurable bicategories having Euler characteristics. If $\B$ is connected, Proposition \ref{pro_2} guarantees that the fiber bicategory $\F$ is determined uniquely up to biequivalence. Then,
\[
\chi(\E)=\chi(\B)\chi(\F).
\]
\end{theorem}
\begin{proof}
By Proposition \ref{pro_1} and Lemma \ref{G}, $\F$ has Euler characteristic.
Our desired formula follows from Lemma \ref{Buck} and Proposition \ref{gr_2}:
\[
\chi(\E) = \chi(\Gr(P^{*})) = \sum_{b \in \ob(\B)} k_b\chi(\F) = \chi(\B)\chi(\F).
\]
\end{proof}

\begin{corollary}
Let $P : \E \to \B$ be fibered and cofibered in pseudogroupoids between measurable bicategories having Euler characteristics. If $\B$ is decomposed as $\coprod \B_{i}$ for connected bicategories $\B_{i}$, we write $\F_i$ as the fiber bicategory over an object of $\B_{i}$. 
Then,
\[
\chi(\E)=\coprod_{i} \chi(\B_{i})\chi(\F_{i}).
\]
\end{corollary}
\begin{proof}
Let $\E_{i}$ denote the full sub-bicategory of $\E$ with the object $P^{-1}(\ob(\B_{i}))$.
That is, the category of morphisms $\E_{i}(x,y)=\E(x,y)$ for $x,y \in \ob(\E_{i})$ and the compositions are the same as $\E$. Then, $\E$ is decomposed as $\coprod_{i} \E_{i}$.
It suffices to verify that the restriction $P_{|\E_{i}} :\E_{i} \to \B_{i}$ is fibered and cofibered in pseudogroupoids.
Obviously, this is locally fibered and cofibered in groupoids. The connectivity and the lifting property of 1-morphisms guarantee that every 1-morphism is cartesian and co-cartesian. 
Hence, the desired formula follows from Theorem \ref{main_2} and Proposition \ref{decompose}:
\[
\chi(\E)=\sum_{i} \chi(\E_{i}) = \sum_{i} \chi(\B_{i})\chi(\F_{i}).
\]

\end{proof}


\begin{thebibliography}{AA999}
\bibitem[BD01]{BD01} J. C. Baez and J. Dolan. From finite sets to Feynman diagrams. \textit{Mathematics unlimited--2001 and beyond.} Springer, Berlin, 2001. pp. 29--50. 
\bibitem[BC03]{BC03}M. Bullejos and A. M. Cegarra. On the geometry of 2-categories and their classifying spaces. \textit{$K$-Theory}  29  (2003), no. 3, 211--229. 
\bibitem[Buc14]{Buc14}M. Buckley. Fibred 2-categories and bicategories. \textit{J. Pure Appl. Algebra} 218  (2014), no. 6, 1034--1074.
\bibitem[CCG10]{CCG10} P. Carrasco, A. M. Cegarra, and A. R. Garzon. Nerves and classifying spaces for bicategories. \textit{Algebr. Geom. Topol.}  10  (2010), no. 1, 219--274.
\bibitem[GNS]{GNS} A. Gonzalez and G. Necoechea and A. Stratmann. Euler characteristic for weak n-categories and $(\infty,1)$-categories. arXiv:1507.06623v1.
\bibitem[Gro71]{Gro71} A. Grothendieck. \textit{Categories fibrees et descente}. Rev\^etements \'etales et groupe fondamental, Lecture Notes in Mathematics, vol. 224, Springer Berlin. pp. 145–194.
\bibitem[Gru06]{Gur06} M. N. Gurski. \textit{An algebraic theory of tricategories}. ProQuest LLC, Ann Arbor, MI, 2006, Thesis (Ph.D.)-The University of Chicago.
\bibitem[Lac04]{Lac04}S. Lack. A Quillen model structure for bicategories. \textit{$K$-Theory} 33 (2004), no. 3, 185--197.
\bibitem[Lei]{Lei} T. Leinster. Basic bicategories. arXiv:math/9810017  . 
\bibitem [Lei08]{Lei08}T. Leinster. The Euler characteristic of a category. \textit{Doc. Math.} 13 (2008), 21--49.
\bibitem [Lei13]{Lei13}T. Leinster. The magnitude of metric spaces. \textit{Doc. Math.} 18 (2013), 857--905.
\bibitem[NT16]{NT16} K. Noguchi and K. Tanaka. The Euler characteristic of an enriched category. \textit{Theor. Appl. Categ.} 31 (2016), 1--30.
\bibitem [Rot64]{Rot64} G. -C. Rota. On the foundations of combinatorial theory. I. Theory of M$\ddot{\mathrm{o}}$bius functions. \textit{Z. Wahrsch. Verw. Gebiete} 2 (1964), 340--368.
\bibitem[Str72]{Str72} R. Street. Fibrations and Yoneda's lemma in a $2$-category.  \textit{Category Seminar} (Proc. Sem., Sydney, 1972/1973),  pp. 104--133. \textit{Lecture Notes in Math}., Vol. 420, Springer, Berlin, 1974.
\bibitem[Str80]{Str80}R. Street. Fibrations in bicategories. \textit{Cahiers Topologie Géom. Différentielle} 21 (1980), no. 2, 111--160.
\bibitem[Tan13]{Tan13}K. Tanaka. A model structure on the category of small categories for coverings. \textit{Math J. Okayama Univ.}, 55 (2013), 95--116.
\bibitem[Wol74]{Wol74}H. Wolff. $V$-cat and $V$-graph. \textit{J. Pure Appl. Algebra}  4  (1974), 123--135.
\end{thebibliography}
\end{document}